\documentclass[11pt]{amsart}
\usepackage{amsfonts}
\usepackage{amsmath}
\usepackage{enumerate}
\usepackage{tikz}

\usepackage{geometry}
\newlength{\ten}
\newcommand{\de}{\partial}

\newcommand{\db}{\overline{\partial}}
\newcommand{\im}{\sqrt{-1}}
\newcommand{\ddbar}{\im\de\db}

\newcommand{\ov}[1]{\overline{#1}}

\newcommand{\Ric}{\mathrm{Ric}}

\newcommand{\PP}{\mathbb P}
\newcommand{\newsort}[1]{}

\numberwithin{equation}{section}

\newtheorem{thm}{Theorem}[section]

\newtheorem{prop}[thm]{Proposition}
\newtheorem{lem}[thm]{Lemma}
\newtheorem{cor}[thm]{Corollary}
\newtheorem{conj}[thm]{Conjecture}
\theoremstyle{definition}
\newtheorem{defn}[thm]{Definition}
\theoremstyle{definition}
\newtheorem{remark}[thm]{Remark}

\begin{document}
\title{Metric Contraction of the Cone Divisor by the Conical K\"ahler-Ricci Flow}
\author{Gregory Edwards}

\begin{abstract}
We use the momentum construction of Calabi to study the conical K\"ahler-Ricci flow on Hirzebruch surfaces with cone angle along the exceptional curve, and show that either the flow Gromov-Hausdorff converges to the Riemann sphere or a single point in finite time, or the flow contracts the cone divisor to a single point and Gromov-Hausdorff converges to a two dimensional projective orbifold. This gives the first example of the conical K\"ahler-Ricci flow contracting the cone divisor to a single point. At the end, we introduce a conjectural picture of the geometry of finite time non-collapsing singularities of the flow on K\"ahler surfaces in general.
\end{abstract}

\today
\maketitle

\section{Introduction}

K\"ahler-Einstein equations with cone singularities along a simple normal crossing divisor have been of contemporary interest to K\"ahler geometers. While such metrics have been around for some time \cite{Do11,Ti96,troy,yau}, they have recently seen a great deal of study and have found important applications to smooth K\"ahler-Einstein metrics \cite{brendle,CGP,CDS1,CDS2,CDS3,DaSo15,GP,JMR,LiSun,SoWa12,Ti}. See also \cite{Ru14} for a general survey of the existing literature.

It is an interesting question to explore the use of parabolic techniques in the conical setting, and one is naturally led to consider the conical K\"ahler-Ricci flow. This is a parabolic flow of conical K\"ahler metrics which deforms the smooth part of the metric by the Ricci tensor, and keeps the conical boundary conditions fixed. In this way, the conical K\"ahler-Ricci flow is a natural generalization of the K\"ahler-Ricci flow to conic metrics. Ideally, one would hope to smoothly deform an initial conical K\"ahler metric to a conical K\"ahler-Einstein metric, if one exists.

The conical K\"ahler-Ricci flow was first studied in the context of Riemann surfaces \cite{Yin10,Yin13}, and there have since been a number of results obtained in that setting \cite{MRS,PSSWa14, PSSWa15}. In higher dimensions the short time existence was shown by Chen-Wang using Bessel functions \cite{ChenWang1} (see also Guo-Song \cite{GuoSo16} for a recent proof using maximum principle arguments), and the long time existence was proved by Shen \cite{Shen,Shen14}.

When studying Hamilton's Ricci flow on Riemannian manifolds \cite{H82}, one of the key features is that while the flow always smoothes the metric for short times, for longer times the non-linearities of the flow can cause the metric to become singular in finite time. In real dimension three, Perelman famously used the technique of performing surgeries to continue the flow past the singularity time \cite{Per3,Per2}.

On a K\"ahler manifold, the Ricci flow is known as the K\"ahler-Ricci flow and was first studied by Cao \cite{cao} using parabolic versions of Yau's estimates \cite{yau}. In this setting, the study of the singularities of the flow becomes much more tractable. Song-Tian \cite{SoTi09} conjectured that on a K\"ahler manifold the only surgeries needed to continue the flow past singularities are algebraic surgeries coming from the minimal model program (MMP) which are bimeromorphic with respect to the underlying complex structure. They also conjectured that the K\"ahler-Ricci flow would carry out an analytic version of the MMP by deforming the initial metric until the flow possibly reaches a finite time singularity. If the volume of the metric tends to zero at the singularity time, the manifold should have the structure of either a Mori fiber space or a Fano manifold; otherwise the volume remains away from zero and one hopes to continue the flow on a new manifold after a bimeromorphic transformation. The bimeromorphic transformation is expected to be precisely that prescribed by \textit{the MMP with scaling of the K\"ahler class} \cite{BiCaHaMc06}, and if such a transformation exists, the flow can be continued in a weak sense on the resulting manifold \cite{SoTi09}. Such bimeromorphic transformations can be made to continue the flow at each successive singularity until either the volume goes to zero in finite time, or the flow reaches a manifold on which it exists for all time. In the latter case the final manifold is a minimal model in the sense of MMP \cite{TiZh,Ts}, and the normalized flow is expected to converge to a unique singular twisted K\"ahler-Einstein metric on its canonical model, possibly of lower dimension. That the K\"ahler-Ricci flow always performs a canonical surgical contraction at finite time non-collapsing singularities is known in dimension two \cite{SoWe11,SoWe11a}, and the Gromov-Hausdorff convergence to a singular K\"ahler-Einstein metric on the canonical model is known in dimensions two and three for minimal models of general type \cite{GuSoWe16,TiZh16}. It is known in general that the singularities of the K\"ahler-Ricci flow always form along an analytic subvariety \cite{CoTo15}.

While the conical K\"ahler-Ricci flow has been studied for many different geometric situations \cite{ChenWang2,E,LiZh,Nom16,Shen,YZh16}, little is known about the geometry of finite time singularities of the flow in general. It is reasonable to expect that singularities of the conical K\"ahler-Ricci flow may behave similarly to those of the K\"ahler-Ricci flow, however the analysis of singularities along the conical K\"ahler-Ricci flow is complicated by the presence of the cone divisor. 

The purpose of this paper is to examine solutions to the conical K\"ahler-Ricci flow on Hirzebruch surfaces with symmetry and to study the geometry of the singularities that occur. In particular, we show that for some initial conical metrics the flow will contract the cone divisor itself to a single point at the singularity time. That the conical K\"ahler-Ricci flow can contract the cone divisor is of interest in its own right, but these finite time singularities also show differences from those of the K\"ahler-Ricci flow on surfaces. First, the variety obtained by the contraction may not be a smooth manifold, but may be only an orbifold in general. Secondly, this demonstrates that the conical K\"ahler-Ricci flow can contract embedded curves of higher negative self-intersection. This is in contrast to the K\"ahler-Ricci flow on surfaces, where the only finite time non-collapsing singularities are contractions of curves of self-intersection $(-1)$ \cite{SoWe11,SoWe11a} and the space obtained after the contraction is always smooth \cite{GH78}. At the end, we state some further conjectures concerning the geometry of more general finite time non-collapsing singularities on K\"ahler surfaces and when we expect them to occur.

\begin{defn}
Let $X$ be a compact $n$-dimensional K\"ahler manifold, and $D$ a smooth irreducible divisor (that is, a compact irreducible codimension one complex submanifold). We say $\omega^*$ is a \textit{conical K\"ahler metric} with \textit{cone angle} $2 \pi \alpha$ ($0<\alpha<1$) along $D$ if $\omega^*$ is a K\"ahler metric on $X \setminus D$, and if for all $p \in D$, $\omega^*$ is quasi-isometric to the model cone metric
\[
	\im \big( \frac {dz^1 \wedge d \ov z^1}{|z^1|^{2(1-\alpha)}} + \sum_{j=2}^n{dz^j \wedge d \ov z^j} \big)
\]
in coordinates $(z^1,...,z^n)$ centered at $p$ such that $D = \{z^1 = 0\}$. We call $\beta = (1-\alpha)$ the \textit{weight} along $D$ and say $\omega^*$ is a conical K\"ahler metric on $(X, \beta D)$. We also write $\mathcal O_X(D)$ for the holomorphic line bundle associated to $D$.

We say $\omega(t)$ is a solution to the conical K\"ahler-Ricci flow starting with $\omega^*$ a conical K\"ahler metric on $(X,\beta D)$ if it satisfies
\begin{equation}\label{eq:1}
	\begin{cases}
		\frac \de {\de t} \omega(t) = - \Ric(\omega(t)) + 2 \pi (1 - \alpha) [D] \\
		\omega(0) = \omega^*
	\end{cases}
\end{equation}
in the sense of currents. Here $[D]$ is the current of integration along $D$, and
\[
	\Ric(\omega(t)) = - \ddbar \log(\omega(t)^n)
\]
is well-defined globally in the sense of currents, and matches the Ricci tensor where $\omega(t)$ is smooth.
\end{defn}

We now state the main theorem of this paper.
\begin{thm}\label{main}
Let $X$ be a Hirzebruch surface of degree $k \geq 1$ with $D_0$ the exceptional curve, and $\omega_0$ a K\"ahler metric on $X$ in the class $2 \pi (b [D_\infty] - a [D_0])$ with $0 < a < b$ satisfying the Calabi ansatz (see Section 2 for relevant definitions), let $\sigma$ be the unique non-zero holomorphic section of $\mathcal O_X(D_0)$ up to scaling, and let $\eta$ be any Calabi invariant Hermitian metric on $\mathcal O_X(D_0)$. 

Then $\omega^* = \omega_0 + \delta \ddbar |\sigma|_\eta^{2 \alpha}$ $(0<\alpha<1)$ is a Calabi invariant conical K\"ahler metric on $X$ with cone angle $2 \pi \alpha$ along $D_0$ for all $\delta>0$ sufficiently small. There are three possible behaviors for the solution to the conical K\"ahler-Ricci flow ~\eqref{eq:1} starting with $\omega^*$:
\begin{enumerate}[(i)]
  \item If
    \begin{equation}\label{eq}
	    0 < \alpha < \mathrm{min}\big(\frac 2 k - (1 + \frac 2 k)\frac a b,1\big),
    \end{equation}
    then $\omega(t)$ contracts the cone divisor $D_0$ to a point at the singularity time $T = \frac {ak} {2 - \alpha k}$.

    Moreover, $\omega(t)$ converges to $\omega_T$ a non-negative current which is a smooth K\"ahler metric on $X \setminus D_0$, and if $(\ov X, d)$ is the metric completion of $(X \setminus D_0, \omega_T)$, then $\ov X$ is homeomorphic to the projective orbifold $\PP^2/\mathbb Z_k$, and $(X,\omega(t))$ converges to $(\ov X,d)$ in the Gromov-Hausdorff sense as $t \rightarrow T^-$.
  \item If
    \begin{equation}\label{eq'}
      \mathrm{max}\big( \frac 2 k - (1 + \frac 2 k)\frac a b, 0 \big) < \alpha < 1,
    \end{equation}
    then $\mathrm{Vol}(X,\omega(t)) \to 0$ as $t \to T^-$ where $T = \frac {b - a}{1 + \alpha}$, and  $(X,\omega(t))$ converges in Gromov-Hausdorff topology to $(\PP^1,\lambda \omega_{\mathrm{fs}})$ where $\lambda = \frac 1 {(1+\alpha)} \big( (k + 2) a + (\alpha k - 2)b \big)$ and $\omega_{\mathrm{fs}}$ is the Fubini-Study metric on $\PP^1$.
  \item If
    \begin{equation}\label{eq''}
      0 < \alpha = \frac 2 k - (1 + \frac 2 k)\frac a b < 1
    \end{equation}
    then $\omega(t)$ exists for all $t < T = \frac {2k}{2 - \alpha k}$ and $(X,\omega(t))$ converges in Gromov-Hausdorff topology to a single point.
\end{enumerate}
\end{thm}

Part (\textit{i}) provides the first explicit example of the conical K\"ahler-Ricci flow contracting the cone divisor at the singularity time, and since $D_0$ has self-intersection $(-k)$, it is the first example of the flow contracting curves of arbitrarily high negative self-intersection on surfaces, in contrast to singularities of the K\"ahler-Ricci flow on surfaces.

We refer to condition ~\eqref{eq} as the \textit{contracting case}, and condition ~\eqref{eq'} as the \textit{collapsing case}. Condition ~\eqref{eq''} corresponds to the log Fano case and is due to Liu-Zhang \cite{LiZh} (cf. Remark \ref{fano}). 

\begin{remark}
In \cite{SoWe11b}, Song-Weinkove study the K\"ahler-Ricci flow with Calabi symmetry on a Hirzebruch surface of degree $k$ starting with an initial K\"ahler metric in the class $2\pi(b [D_\infty] - a [D_0])$ with $0 < a < b$, and prove that if $k=1$ and $3a<b$, then the K\"ahler-Ricci flow contracts the exceptional curve to a single point at the singularity time and the flow converges to a compact metric space homeomorphic to $\PP^2$. Otherwise, the volume of $X$ converges to zero at the singularity time and the flow converges to either a single point, if $k=1$ and $3a=b$; or in all other cases, in particular whenever $k \geq 2$, the flow converges to $(\PP^1,\lambda \omega_{\mathrm{fs}})$ for some explicit constant $\lambda>0$ depending only on the initial K\"ahler class. This behavior of the K\"ahler-Ricci flow had been conjectured earlier by Feldman-Ilmanen-Knopf \cite{FIK}.

If we set $\alpha = 1$ in Theorem ~\ref{main}, so the conical K\"ahler-Ricci flow reduces to the smooth K\"ahler-Ricci flow, we recover the conditions for contracting the exceptional divisor in the smooth setting as in \cite{SoWe11b}.
\end{remark}

\begin{remark}\label{fano}
If $\alpha = \frac 2 k - (1 + \frac 2 k) \frac a b$ then 
\[
	[K^{-1}_X] -(1-\alpha)[D_0] 
\]
is a K\"ahler class (cf. equation ~\eqref{eq:12}) and the initial K\"ahler class is a positive multiple of it. Thus the conical K\"ahler-Ricci flow will converge in Gromov-Hausdorff topology to a single point at the singularity time as proved by Liu-Zhang \cite{LiZh}. 
\end{remark}

\begin{remark}
  Although Theorem \ref{main} is stated for Hirzebruch surfaces, the proof can be generalized to arbitrary projective line bundles of positive degree over $\PP^{n-1}$ for $n \geq 2$ with few changes. Since the phenomena are already apparent in dimension two, we give the proof for Hirzebruch surfaces and leave the statement of the theorem in higher dimensions up to the reader.
\end{remark}

The organization of the rest of the paper is as follows. In Section 2 we define Hirzebruch surfaces and the Calabi ansatz. In Section 3 we review relevant estimates along the conical K\"ahler-Ricci flow without symmetry. In Section 4 we use the Calabi ansatz to reduce the conical K\"ahler-Ricci flow with symmetry to a scalar parabolic equation. In Section 5 we prove estimates along the flow with symmetry in the contracting case and prove Theorem \ref{main}.\textit{i}. In Section 6 we prove estimates for the flow with symmetry in the collapsing case and prove Theorem \ref{main}.\textit{ii}. In Section 7 we state some further conjectures for the behavior of finite time non-collapsing singularities of the flow on K\"ahler surfaces in general, and in Section 8 we outline an example where we expect the Gromov-Hausdorff limit along the conical K\"ahler-Ricci flow to be a metric with cone singularities along a divisor without simple normal crossing support.

\section{Hirzebruch surfaces and Calabi invariant K\"ahler metrics}

Consider the natural action of $U(2)$ on $\PP^2$ fixing a single point which we can take to be $[1:0:0]$. The subgroup $\mathbb Z_k \subseteq U(2)$ acts on $\PP^2$ by
\[
  [x_0:x_1:x_2] \mapsto [x_0:\zeta_k x_1 : \zeta_k x_2]
\]
where $\zeta_k$ is a $k^{\mathrm{th}}$-root of unity. The $\mathbb Z_k$-action has an isolated fixed point at $[1:0:0]$, and $\PP^2/\mathbb Z_k$ is smooth away from a single orbifold singularity. The blow-up centered at this point is smooth and the blow-up map,
\begin{equation}\label{eq:14}
	\pi: X \rightarrow \PP^2/\mathbb Z_k,
\end{equation}
is a biholomorphism on $X \setminus D_0$, where $D_0$ is the exceptional curve.

The $U(2)/\mathbb Z_k$-action on $\PP^2/\mathbb Z_k$ naturally lifts to $X$, and forms a maximal compact subgroup of its automorphism group \cite{Ca82}. A K\"ahler metric is said to satisfy the \textit{Calabi ansatz} if it is invariant under the $U(2)/\mathbb Z_k$-action on $X$.

$X$ is then a smooth projective variety, and can be given the structure of a ruled surface over $\PP^1$. That is, there is a map
\begin{equation}\label{eq:15}
 p: X \rightarrow \PP^1
\end{equation}
such that all the fibers are smooth and isomorphic to $\PP^1$. It follows that $X$ is isomorphic to the projectivization of a rank 2 holomorphic vector bundle on $\PP^1$ \cite{GH78}, and
\[
	X \cong \PP( \mathcal O_{\PP^1} \oplus \mathcal O_{\PP^1}(-k) ).
\]
for some $k \geq 1$ (we omit the case where $k=0$ and $X$ is biholomorphic to $\PP^1 \times \PP^1$). The complex subvariety corresponding to the zero section of $\mathcal O_{\PP^1}(-k)$ is identified with the exceptional curve in the blow-up construction, and has self-intersection number $D_0 \cdot D_0 = -k$. Similarly, the subvariety corresponding to the zero section of $\mathcal O_{\PP^1}$, which we call the \textit{infinity section}, $D_\infty$, has self-intersection number $D_\infty \cdot D_\infty = k$.

Moreover, $H^{1,1}(X,\mathbb R)$ is spanned by the Poincar\'e duals of the zero section and the infinity section \cite{GH78}. Thus any class $\xi \in H^{1,1}(X,\mathbb R)$ can be written in the form
\[
	\xi = b [D_\infty] - a [D_0]
\]
for real constants $a,b$. The class admits a K\"ahler metric if and only if $0 < a < b$.

Next, we construct $U(2)/\mathbb Z_k$-invariant metrics on $X$ by working on the total space of the line bundle $\mathcal O_{\PP^1}(-k)$ which is biholomorphic to $X \setminus D_\infty$. The line bundle is covered by two charts $U$ and $U'$ biholomorphic to $\mathbb C^2$. We let $z$ and $w$ (resp. $z'$ and $w'$) be the coordinates on $U$ with $z$ the base coordinate, and $w$ the fiber coordinate (resp. on $U'$ with $z'$ and $w'$ the base and fiber coordinates), so that $D_0 = \{ w = 0 \}$ in this chart. The transition function on the intersection is given by
\[
	\Psi_{UU'}: (U \cap U') \to (U \cap U')
\]
\[
	\Psi_{UU'}(z,w) = (\frac 1 z, w z^k) := (z',w').
\]
Let $h$ be a Hermitian metric on $\mathcal O_{\PP^1}(-k)$ with Chern curvature form
\[
	\mathrm{curv}(h) := -\ddbar \log h =  -k \omega_{\mathrm{fs}}
\]
where $\omega_{\mathrm{fs}}$ is the Fubini-Study metric on $\PP^1$ satisfying
\[
	\begin{cases}
		\Ric(\omega_{\mathrm{fs}}) = 2 \omega_{\mathrm{fs}} \\
		\int_{\PP^1}{\omega_{\mathrm{fs}}} = 2 \pi.
	\end{cases}
\]
We define 
\[
	\rho = \log |w|^2_h
\]
to be our $U(2)/\mathbb Z_k$-invariant coordinate, which is well-defined on $X \setminus (D_0 \cup D_\infty)$. Then for any K\"ahler metric on $X$ satisfying the Calabi ansatz, the local potential on $X \setminus (D_0 \cup D_\infty)$ is of the form
\begin{align*}
	\omega 
	& = \ddbar v(\rho) \\
	& = v'(\rho) \ddbar \rho + \im v''(\rho) \de \rho \wedge \db \rho \\
	& = k v'(\rho) p^*\omega_{\mathrm{fs}} + \im v''(\rho) (\frac {dw} w + \frac {\de h} h) \wedge (\frac {d \ov w} {\ov w} + \frac {\db h} h)
\end{align*}
and any such form is positive definite if and only if $v'(\rho) > 0$ and $v''(\rho)>0$.

\begin{lem}[Calabi \cite{Ca82}]\label{calabi}
If $\omega = \ddbar v(\rho)$ on $X \setminus (D_0 \cup D_\infty)$ with $v'>0$ and $v''>0$, then $\omega$ can be extended to a smooth K\"ahler metric on $X$ satisfying the Calabi ansatz if and only if there exist smooth functions
\[
	v_0, v_\infty : [0,\infty) \rightarrow \mathbb R
\]
with $v_0'(0)>0$, and $v_\infty'(0)>0$ such that
\[
	v_0(e^\rho) = v(\rho) - a \rho \text{, and } v_\infty(e^{-\rho}) = v(\rho) - b \rho
\]
for real constants $0 < a < b$.
\end{lem}

If $\omega$ is constructed as above, one can check using Poincar\'e duality that the K\"ahler class is given by
\[
	[\omega] = 2 \pi (b [D_\infty] - a [D_0]).
\]

\section{Estimates along the conical K\"ahler-Ricci flow}

We now recall some of the main results concerning the unnormalized conical K\"ahler-Ricci flow using smooth approximations.

First, we remark that for any K\"ahler metric $\omega_0$ and $\sigma$ a non-trivial section of a holomorphic line bundle, equipped with a Hermitian metric $\eta$, whose vanishing locus defines a smooth divisor,
\[
	\omega^* = \omega_0 + \delta \ddbar |\sigma|^{2\alpha}_\eta
\]
is a conical K\"ahler metric with cone angle $2 \pi \alpha$ along $D = \{ \sigma = 0 \}$ for all sufficiently small $\delta>0$ \cite{Ru14}.

Furthermore, following the method of Campana-Guenancia-P\u{a}un \cite{CGP}, we can approximate such conical metrics by smooth K\"ahler metrics of the form
\[
  \omega_\epsilon = \omega_0 + \delta \ddbar \chi( |\sigma|^2_\eta + \epsilon^2)
\]
where $\chi = \chi_{\alpha,\epsilon}: [\epsilon^2,\infty) \to \mathbb R$ is the function defined by
\[
 \chi( t + \epsilon^2) = \frac 1 \alpha \int_0^t \frac {(r + \epsilon^2)^\alpha - \epsilon^{2 \alpha}} r dr,
\]
and the $\omega_\epsilon$ are smooth K\"ahler metrics for all $\epsilon>0$, satisfy $\omega_\epsilon \geq \gamma \omega_0$ for some uniform constant $\gamma >0$, and converge to $\omega^*$ in the sense of currents and in $C^\infty_{\mathrm{loc}}(X \setminus D)$ as $\epsilon$ tends to zero \cite{CGP,GP}. 

Liu-Zhang first used this method of smooth approximation to study the conical K\"ahler-Ricci flow on log Fano manifolds \cite{LiZh}, and Shen adopted this technique to obtain the following long time existence result for the unnormalized conical K\"ahler-Ricci flow.

\begin{thm}[Shen \cite{Shen}]
	If $\omega^* = \omega_0 + \delta \ddbar |\sigma|^{2 \alpha}_\eta$ is a conical K\"ahler metric, then a unique maximal solution to the conical K\"ahler-Ricci flow starting with $\omega^*$ exists on $[0,T)$ where
	\[
		T = \mathrm{sup} \{ t > 0 \big| [\omega_0] - t (c_1(X) - (1 - \alpha) [D]) \text{ is a K\"ahler class} \}
	\]
	and the solution is approximated, as $\epsilon$ tends to zero, by a sequence of smooth twisted K\"ahler-Ricci flows
	\begin{equation}\label{eq:2}
		\begin{cases}
			\frac \de {\de t} \omega_\epsilon(t) = - \Ric(\omega_\epsilon(t)) + (1-\alpha) \big( \ddbar \log ( |\sigma|_\eta^2 + \epsilon^2) + \mathrm{curv}(\eta) \big) \\
			\omega_\epsilon(0) = \omega_\epsilon := \omega_0 + \delta \ddbar \chi (|\sigma|^2_\eta + \epsilon^2),
		\end{cases}
	\end{equation}
	and the convergence is globally in the sense of currents and in $C^\infty_{\mathrm{loc}}((X \setminus D) \times [0,T))$.
\end{thm}

Note that if $\sigma$ is the holomorphic section defining $D_0$, and $\eta$ is any Calabi invariant Hermitian metric on $\mathcal O_X(D_0)$, then $|\sigma|^2_\eta$ is a $U(2)/ \mathbb Z_k$-invariant function. Thus, if $\omega_0$ is a K\"ahler metric satisfying the Calabi ansatz, then $\omega_\epsilon$ also satisfies the Calabi ansatz. Moreover, the terms on the right hand side of equation ~\eqref{eq:2} are also $U(2)/\mathbb Z_k$-invariant, so $\omega_\epsilon(t)$ will satisfy the Calabi ansatz as long as the flow exists.

Now, if $\omega_\epsilon(t)$ is a solution to the twisted K\"ahler-Ricci flow ~\eqref{eq:2}, then the K\"aher class evolves by
\[
	[\omega_\epsilon(t)] = [\omega_0] - 2 \pi t \big( c_1(X) - (1-\alpha) [D_0] \big),
\]
where
\begin{equation}\label{eq:12}
	c_1(X) = \big( (1 + \frac 2 k) [D_\infty] + (1 - \frac 2 k) [D_0] \big).
\end{equation}
is the first Chern class of $X$ \cite{GH78}.

Thus we can write
\[
	[\omega_\epsilon(t)] = 2 \pi \big( b_t [D_\infty] - a_t [D_0])
\]
where
\[
	b_t = b - (1 + \frac 2 k) t
\]
and
\[
	a_t = a - (\frac 2 k - \alpha) t.
\]

In particular, if $\alpha$ and $[\omega_0]$ satisfies condition ~\eqref{eq} of Theorem \ref{main}, then $a_t \rightarrow 0$ as $t \rightarrow T^-$, while $b_t > b_T > 0$ for all $t \leq T$; and if $\alpha$ and $[\omega_0]$ satisfy condition ~\eqref{eq'}, then $a_t > a_T > 0$ and $(b_t - a_t) \to 0$ as $t \to T^-$.

Next, we recall some further estimates along the conical K\"ahler-Ricci flow.

\begin{prop}[Shen \cite{Shen}]\label{CKRF1}
If $\omega_\epsilon(t)$ is a solution to the twisted K\"ahler-Ricci flow ~\eqref{eq:2}, then for any smooth volume form $\Omega$, there exists a uniform $C>0$ independent of $\epsilon$ such that
\[
	\omega_\epsilon(t)^2 \leq C \frac \Omega {|\sigma|_\eta^{2(1-\alpha)}}
\]
\end{prop}

\begin{proof}
This estimate is proved in Section 2.1 of \cite{Shen}, and is equivalent to the uniform upper bound on the time-derivative of the potential.
\end{proof}

We have also the following second order estimate.

\begin{prop}\label{CKRF2}
Let $\omega(t)$ be a solution to the conical K\"ahler-Ricci flow on a Hirzebruch surface $X$ satisfying condition ~\eqref{eq} of Theorem ~\ref{main}. Then as $t \rightarrow T^-$, $\omega(t)$ converges weakly to a non-negative current $\omega_T$ which is a smooth K\"ahler metric on $X \setminus D_0$ and the convergence is in $C^\infty_{\mathrm{loc}}(X \setminus D_0)$.
\end{prop}

\begin{proof}
This follows from the proof of Theorem 1.3 in \cite{Shen}. Indeed under our assumptions
\[
  [\omega(t)] - \lambda [D_0]
\] 
is a K\"ahler class for all $t\in[0,T)$ for every $\lambda>0$ sufficiently small, and the stable base locus of $[\omega_0] - 2 \pi T(c_1(X) - (1-\alpha)[D_0])$ is equal to $D_0$.
\end{proof}

\section{The conical K\"ahler-Ricci flow with Calabi symmetry}

Now we use Calabi symmetry to reduce the twisted K\"ahler-Ricci flows ~\eqref{eq:2} to a scalar parabolic equation.

First, we remark that $\mathrm{dim}_{\mathbb C} H^0(X,\mathcal O_X(D_0)) = 1$ so that up to scaling there is a unique non-trivial global holomorphic section of this line bundle and this section vanishes to first order along $D_0$.

Moreover, in the coordinate chart $U$ the function $\sigma(z,w) = w$ is holomorphic, vanishes to first order along $D_0$, and can be extended to a global holomorphic section of $\mathcal O_{X}(D_0)$.

Next, we define an explicit $U(2)/\mathbb Z_k$-invariant Hermitian metric on $\mathcal O_X(D_0)$ by 
\[
	\eta_0 = \frac h {1 + e^\rho}.
\]
Then
\[
	|\sigma|_{\eta_0}^2 = \frac {e^\rho}{1 + e^\rho}
\]
is a well-defined $U(2)/\mathbb Z_k$-invariant $C^\infty$ function on $X$, and so $\eta_0$ is indeed a Hermitian metric on this line bundle.

If $\eta$ is any other $U(2)/\mathbb Z_k$-invariant Hermitian metric on this line bundle, then
\[
  \eta = e^{-\psi} \eta_0
\]
for some smooth globally defined function $\psi=\psi(\rho)$. In particular, $\psi$ and all its derivatives are uniformly bounded. Up to rescaling $\eta$, we can always assume that $\psi(0) = 0$.

We calculate
\[
	|\sigma|^2_\eta = \frac {e^{-\psi+\rho}}{1 + e^\rho},
\]
and
\[
	\mathrm{curv}(\eta) = - \ddbar \rho + \ddbar \log ( 1 + e^\rho) + \ddbar \psi.
\]
So on $X \setminus (D_0 \cup D_\infty)$,
\[
 \ddbar \log(|\sigma|^2_\eta + \epsilon^2) + \mathrm{curv}(\eta) = \ddbar \Big( \log \big( e^{-\psi + \rho} + \epsilon^2 e^\rho + \epsilon^2 \big) - \rho + \psi \Big).
\]

Next, if $\omega = \ddbar v(\rho)$ is a smooth K\"ahler metric satisfying the Calabi ansatz, then we can calculate the Ricci curvature as follows:
\begin{align*}
	\Ric(\omega) 
	& = - \ddbar \log (\omega^2) \\
	& = - \ddbar \log \big( k v'(\rho) v''(\rho) p^* \omega_{\mathrm{fs}} \wedge \im \frac {dw \wedge d \ov w}{|w|^2} \big) \\
	& = - \ddbar \log v'(\rho) - \ddbar \log v''(\rho) + \frac 2 k \ddbar \rho.
\end{align*}

Thus $\omega_\epsilon(t) = \ddbar v_\epsilon(t)$ solves the smooth twisted K\"ahler-Ricci flow ~\eqref{eq:2} if and only if $v_\epsilon(t)$ solves the parabolic equation
\begin{equation}\label{eq:4}
	\begin{cases}
	\frac \de {\de t} v_\epsilon(t,\rho) = \log v'_\epsilon + \log v_\epsilon'' - \frac 2 k \rho + (1 - \alpha) \log \big( 1 + \epsilon^2 e^\psi(1 + e^{-\rho}) \big) + c_t \\
	v_\epsilon(0,\rho) =  v(\rho) + \delta \chi(|\sigma|^2_\eta + \epsilon^2)
	\end{cases}
\end{equation}
where
\[
  c_t = - \log v_\epsilon'(t,0) - \log v_\epsilon''(t,0) - (1-\alpha) \log (1 + 2 \epsilon^2)
\]
is chosen so that
\[
	v_\epsilon(t,0) = 0
\]
for all $t \in [0,T)$.

We calculate the following evolution equations for $v_\epsilon$:
\begin{equation}\label{eq:5}
	\frac \de {\de t} v_\epsilon' = - \frac 2 k + \frac {v_\epsilon''}{v_\epsilon'} + \frac {v_\epsilon'''}{v_\epsilon''} + (1-\alpha)\Big( \log \big( 1 + \epsilon^2 e^{\psi}(1 + e^{-\rho}) \big) \Big)',
\end{equation}
\begin{equation}\label{eq:6}
	\frac \de {\de t} v_\epsilon'' 
	= - \frac {(v_\epsilon'')^2}{(v_\epsilon')^2} + \frac {v_\epsilon'''}{v_\epsilon'} - \frac {(v_\epsilon''')^2}{(v_\epsilon'')^2} + \frac {v_\epsilon^{(4)}}{v_\epsilon''} + (1-\alpha)\Big( \log \big( 1 + \epsilon^2 e^{\psi} (1 + e^{-\rho}) \big) \Big)'',
\end{equation}
and
\begin{align}\label{eq:7}
	\frac \de {\de t} v_\epsilon''' = \frac {v_\epsilon^{(4)}}{v_\epsilon'}  - 3 \frac{v_\epsilon'' v_\epsilon'''}{(v_\epsilon')^2} + 2 \frac {(v_\epsilon'')^3}{(v_\epsilon')^3} & + \frac {v_\epsilon^{(5)}}{v_\epsilon''} - 3 \frac {v_\epsilon'''v_\epsilon^{(4)}}{(v_\epsilon'')^2} + 2 \frac {(v_\epsilon''')^3}{(v_\epsilon'')^3} \\
	& + (1-\alpha) \Big( \log \big( 1 + \epsilon^2 e^\psi(1 + e^{-\rho}) \big) \Big)'''. \nonumber
\end{align}

To control the excess terms in these equations we make use of the following lemma.

\begin{lem}\label{eq:9}
 Let
\[
 \theta = 1 + \epsilon^2 e^\psi + \epsilon^2 e^{\psi - \rho},
\]
 then there exists a uniform constant $C>0$, independent of $\epsilon$ and depending only on $\psi$, such that
\[
 \big| (\log \theta)' \big| + \big| (\log \theta)'' \big| + \big| ( \log \theta )''' \big| \leq C
\]
\end{lem}
\begin{proof}
First, we note that $\theta$ is bounded from below,
\begin{equation}\label{eq:3}
 \theta = 1 + \epsilon^2 e^\psi + \epsilon^2 e^{\psi - \rho} \geq \epsilon^2 C^{-1} (1 + e^{-\rho}).
\end{equation}

Next, to estimate
\[
	(\log \theta)' = \frac {\theta'}{\theta},
\]
\[
 (\log \theta)'' = \frac {\theta''}{\theta} - \frac {(\theta')^2}{\theta^2},
\]
and
\[
 ( \log \theta )''' = \frac{\theta'''}{\theta} - 3 \frac{\theta''\theta'}{\theta^2} + 2 \frac {(\theta')^3}{\theta^3},
\]
it suffices to bound $|\theta'/\theta|$, $|\theta''/\theta|$, and $|\theta'''/\theta|$. We calculate
\[
 \theta' = \epsilon^2 e^\psi (1 + e^{-\rho}) \psi' - \epsilon^2 e^\psi e^{-\rho},
\]
\[
 \theta'' = \epsilon^2 e^\psi(1 + e^{-\rho})( \psi'' + (\psi')^2) + \epsilon^2 e^\psi e^{-\rho}(1 - 2 \psi'),
\]
and
\[
 \theta''' = \epsilon^2 e^{-\psi}(1 + e^{-\rho}) (\psi''' + 3 \psi''\psi' + (\psi')^3) + \epsilon^2 e^\psi e^{-\rho}( - \psi'' - 3(\psi')^2  + 3 \psi' -1),
\]
so that
\begin{align*}
	|\theta'| & \leq \epsilon^2 C (1 + e^{-\rho})|\psi'| \\
	& \leq \epsilon^2 C (1 + e^{-\rho}),
\end{align*}
\begin{align*}
	|\theta''| & \leq \epsilon^2 C (1 + e^{-\rho})( |\psi''| + |\psi'|^2) + \epsilon^2 C e^{-\rho}(|\psi'| + 1) \\
	& \leq \epsilon^2 C (1 + e^{-\rho}),
\end{align*}
and
\begin{align*}
	|\theta'''| & \leq \epsilon^2 C (1 + e^{-\rho})(|\psi'''| + |\psi''||\psi'| + |\psi'|^3) + \epsilon^2 C e^{-\rho} (|\psi''| + |\psi'| + 1) \\
	& \leq \epsilon^2 C(1+e^{-\rho}).
\end{align*}
Combined with ~\eqref{eq:3}, this proves the estimates. \end{proof}

\section{Estimates Along the Flow with Symmetry in the Contracting Case}

In this section we prove estimates on the evolving potential under the assumption that the initial K\"ahler class satisfies condition ~\eqref{eq} of Theorem \ref{main}. In particular, we have that $a_t \to 0$ as $t \to T^-$ and $b_t - a_t > b_T - a_T > 0$.

We first prove estimates on the derivative of the potential. The proof is similar to Lemma 4.4 in \cite{SoWe11b}.

\begin{lem}\label{lemma1}
There is a uniform constant $C>0$ independent of $\epsilon$ such that
\[
  v'_\epsilon \leq C e^{\alpha \rho/2} + a_t
\]
\end{lem}

\begin{proof}

Define a smooth reference metric
\begin{equation}\label{eq:8}
	\widehat \omega = \ddbar \widehat v(\rho)
\end{equation}
where
\begin{equation}\label{eq:8'}
	\widehat v(\rho) = a \rho + (b-a) \log(e^\rho + 1).
\end{equation}
Then $\widehat v$ is smooth,
\[
	\widehat v'(\rho) = a + (b-a) \frac {e^\rho}{1 + e^\rho},
\]
and
\[
	\widehat v''(\rho) = (b-a) \frac {e^\rho}{(1 + e^\rho)^2}.
\]
We have
\[
	0 < a < \widehat v'(\rho) < b,
\]
\[
	0 < \widehat v''(\rho) < (b-a)e^\rho,
\]
and $\widehat \omega$ is a smooth K\"ahler metric in the same class as $\omega_0$. Thus
\begin{align*}
	\widehat \omega^2 
	& = \widehat v'(\rho) \widehat v''(\rho) \ddbar \rho \wedge \im \de \rho \wedge \db \rho \\
	& \leq b (b-a) e^\rho \ddbar \rho \wedge \im \de \rho \wedge \db \rho
\end{align*}
is a smooth volume form, and so by Proposition \ref{CKRF1} there is a uniform constant such that
\[
	\frac {\omega_\epsilon(t)^2}{\widehat \omega^2} = \frac {v_\epsilon' v_\epsilon''}{ b(b-a)} e^{-\rho} \leq \frac C {|\sigma|_\eta^{2(1-\alpha)}} \leq C e^{-(1-\alpha) \rho}
\]
Hence
\[
 v_\epsilon' v_\epsilon'' = \frac 1 2 ((v_\epsilon')^2)' \leq C e^{\alpha \rho},
\]
and therefore
\[
	(v_\epsilon')^2 - a_t^2 \leq C e^{\alpha \rho},
\]
from which the desired inequality follows.
\end{proof}

The previous lemma could be improved for $\rho$ large by proving a similar estimate for $(b_t - v_\epsilon')$, however we are primarily concerned with the behavior of $v_\epsilon$ as $\rho \to - \infty$ since the metric is already bounded away from $D_0$.

\begin{lem}\label{lemma2}
There is a uniform constant $C>0$ independent of $\epsilon$ such that
\[
  v''_\epsilon \leq C (v'_\epsilon-a_t)(b_t - v'_\epsilon) \leq C e^{\alpha \rho/2}
\]
\end{lem}

While the conclusion of this lemma is similar to Lemma 4.5 in \cite{SoWe11b}, the extra terms in our parabolic equation result in additional difficulties not present is the proof given there.

\begin{proof}

Define
\[
	H_\epsilon = \log \frac {v_\epsilon''}{(v'_\epsilon-a_t)(b_t - v'_\epsilon)}.
\]

We claim that for each $\epsilon>0$ and $t \in [0,T)$ fixed, $H_\epsilon$ is bounded from above as $\rho \rightarrow \pm \infty$. Indeed, since $\omega_\epsilon(t)$ remains smooth for all $t < T$, by Lemma \ref{calabi} there exist smooth functions $v_{0,\epsilon},v_{\infty,\epsilon} : [0,\infty) \times [0,T) \rightarrow \mathbb R$ with $v'_{0,\epsilon}(0,t)>0$ and $v'_{\infty,\epsilon}(0,t)>0$ such that
\begin{align*}
	v_\epsilon(\rho,t) 
	& = v_{0,\epsilon}(e^\rho,t) + a_t \rho \\
	& = v_{\infty,\epsilon}(e^{-\rho},t) + b_t \rho
\end{align*}
so that e.g. as $\rho$ tends to $-\infty$,
\begin{align*}
	e^{H_\epsilon} & = \frac {v''_\epsilon}{(v'_\epsilon - a_t)(b_t - v'_\epsilon)} \\
	& = \frac {e^{2 \rho}v_{0,\epsilon}'' + e^\rho v_{0,\epsilon}'}{e^\rho v_{0,\epsilon}'(b_t - a_t - e^\rho v'_{0,\epsilon})} \\
	& = \frac 1 {(b_t - a_t - e^\rho v'_{0,\epsilon})} + \frac {e^\rho v_{0,\epsilon}''}{v_{0,\epsilon}' (b_t - a_t - e^\rho v_{0,\epsilon}')} \\
	& \leq \frac 1 {(b_t - a_t)} + 1
\end{align*}
where we used that $v_{0,\epsilon}'(0)>0$ in the final inequality to obtain that the second term is converging to zero. The estimate as $\rho$ tends to $+ \infty$ is similar.

Next, we calculate
\begin{equation}
H_\epsilon' = \frac {v_\epsilon'''}{v_\epsilon''} - \frac {v_\epsilon''}{v_\epsilon' - a_t} + \frac {v_\epsilon''}{b_t - v_\epsilon'},
\end{equation}
\begin{equation}
H_\epsilon'' = \frac {v_\epsilon^{(4)}}{v_\epsilon''} - \frac {(v_\epsilon''')^2}{(v_\epsilon'')^2} - \frac {v_\epsilon'''}{v_\epsilon'-a_t} + \frac {(v_\epsilon'')^2}{(v_\epsilon' - a_t)^2} + \frac {v_\epsilon'''}{b_t - v_\epsilon'} + \frac {(v_\epsilon'')^2}{(b_t - v_\epsilon')^2},
\end{equation}
and, making use of equations ~\eqref{eq:5}, ~\eqref{eq:6}, and Lemma \ref{eq:9},
\begin{align*}
	\frac {\de H_\epsilon}{\de t} \leq & \frac 1 {v''_\epsilon} \big( - \frac {(v_\epsilon'')^2}{(v_\epsilon')^2} + \frac {v_\epsilon'''}{v_\epsilon'} - \frac {(v_\epsilon''')^2}{(v_\epsilon'')^2} + \frac{v_\epsilon^{(4)}}{v_\epsilon''} + (1-\alpha)C(\psi) \big) \\
	& - \frac  1 {v'_\epsilon - a_t} \big( \frac {v_\epsilon''}{v_\epsilon'} + \frac {v_\epsilon'''}{v_\epsilon''} - (1-\alpha)C(\psi) - \alpha \big) \\
	& + \frac 1 {b_t - v'_\epsilon} \big( \frac {v_\epsilon''}{v_\epsilon'} + \frac {v_\epsilon'''}{v_\epsilon''} + (1-\alpha)C(\psi) + 1 \big) \\
	\leq & \frac {H_\epsilon'}{v_\epsilon'} + \frac {H_\epsilon''}{v_\epsilon''} - \frac {v_\epsilon''}{(v_\epsilon')^2} + (1 - \alpha)C(\psi)\frac 1 {v_\epsilon''} - \frac {v_\epsilon''}{(v_\epsilon'-a_t)^2} - \frac {v_\epsilon''}{(b_t - v_\epsilon')^2} \\
	& + \frac 1 {v_\epsilon' - a_t} \big( 1 + (1-\alpha)C(\psi) \big) + \frac 1 {b_t - v_\epsilon'} \big( 1 + (1-\alpha)C(\psi) \big).
\end{align*}

Now fix $T' < T$. If $H_\epsilon$ achieves its maximum at $(x_0,t_0) \in \mathbb R \times (0,T']$, then by the parabolic maximum principle
\[
	H_\epsilon'(x_0,t_0) = 0 \text{, }	H_\epsilon''(x_0,t_0) \leq 0 \text{, and } \frac {\de H_\epsilon}{\de t}(x_0,t_0) \geq 0.
\]
Thus we obtain
\begin{align*}
	0 & \leq v_\epsilon''(x_0,t_0) \frac {\de H_\epsilon}{\de t}(x_0,t_0) \\
	& \leq C(\alpha,\psi) + C(\alpha,\psi) \big( \frac{v_\epsilon''}{v_\epsilon' - a_t} + \frac {v_\epsilon''}{b_t - v_\epsilon'} \big) - \big( \frac {(v_\epsilon'')^2} {(v_\epsilon' - a_t)^2} + \frac {(v_\epsilon'')^2} {(b_t - v_\epsilon')^2} \big) \\
	& \leq C(\alpha,\psi) + C(\alpha,\psi) \frac {v_\epsilon''}{(v'_\epsilon-a_t)(b_t - v_\epsilon')} - C^{-1} \frac {(v_\epsilon'')^2}{(v_\epsilon' - a_t)^2(b_t - v_\epsilon')^2}.
\end{align*}

In particular, since there is a uniform $C = C(\alpha,\psi)>0$ independent of $\epsilon$ such that 
\[
	0 \leq C + C e^{H_\epsilon}(x_0,t_0) - C^{-1} e^{2H_\epsilon}(x_0,t_0)
\]
we get that at the point of maximum
\[
	e^{H_\epsilon}(x_0,t_0) \leq C
\]
and therefore
\[
	H_\epsilon \leq C
\]
on $\mathbb R \times [0,T']$ independent of $\epsilon$. Our estimates are independent of $T'$, so letting $T'$ tend to $T$ we obtain the uniform bound on $\mathbb R \times [0, T)$, which gives the desired estimate.
\end{proof}

Finally, we prove the Gromov-Hausdorff convergence for the conical K\"ahler-Ricci flow with symmetry in the contracting case.

\begin{lem}
  If $\omega(t)$ solves the conical K\"ahler-Ricci flow ~\eqref{eq:1} and satisfies condition ~\eqref{eq} of Theorem \ref{main}, then $\omega(t)$ converges to $\omega_T$ a closed non-negative current which is a smooth K\"ahler metric on $X \setminus D_0$, and if $(\ov X, d)$ is the metric completion of $(X \setminus D_0, \omega_T)$, then $\ov X$ is homeomorphic to $\PP^2/\mathbb Z_k$, and $(X,\omega(t))$ converges to $(\ov X, d)$ in the Gromov-Hausdorff topology.
\end{lem}

\begin{proof}
From Lemma \ref{lemma1} and Lemma \ref{lemma2} it follows that
\begin{align*}
	\omega_\epsilon(t) & = v'_\epsilon(\rho) \ddbar \rho + v''_\epsilon(\rho) \im \de \rho \wedge \db \rho \\
	& \leq a_t \ddbar \rho + C e^{\alpha \rho/2} \big( k \ddbar \rho + \im \de \rho \wedge \db \rho \big)
\end{align*}
where $a_t \rightarrow 0$ as $t \rightarrow T^-$ and $C>0$ is independent of $\epsilon$.

We let $\epsilon_i \rightarrow 0^+$ be a sequence such that $\omega_{\epsilon_i}(t)$ converges weakly to $\omega(t)$ the unique solution of the conical K\"ahler-Ricci flow on $X \times[0,T)$. Then $\omega(t)$ satisfies
\[
 \omega(t) \leq a_t \ddbar \rho + C e^{\alpha \rho /2}(k \ddbar \rho + \im \de \rho \wedge \db \rho).
\]
Define
\[
	\widetilde \omega_T = e^{\alpha \rho/2} (k \ddbar \rho + \im \de \rho \wedge \db \rho).
\]
Then by Proposition ~\ref{CKRF2}, as $t \rightarrow T^-$, $\omega(t)$ converges to a closed non-negative current $\omega_T$ which is a smooth K\"ahler metric on $X \setminus D_0$ and satisfies
\[
	\omega_T \leq C \widetilde \omega_T. 
\]

Since we have smooth convergence on compact subsets of $X \setminus D_0$, to determine the metric completion of $(X \setminus D_0,\omega_T)$ it suffices to determine the metric completion of $(W_\delta \setminus D_0,\omega_T)$ where
\[
	W_\delta = \{ |\sigma|_{\eta_0}^2 = \frac {e^\rho}{1 + e^\rho} < \delta \}
\]
is an open neighborhood of $D_0$ for any fixed $0 < \delta < 1$.

From the birational map ~\eqref{eq:14} we obtain a $k:1$ covering
\[
	\Phi: \mathbb C^2 \setminus \{ 0 \} \rightarrow X \setminus (D_\infty \cup D_0).
\]
The map can be described explicitly onto the charts $U$ and $U'$ (cf. Section 2) for
\[
 \Phi_U: \mathbb C^2 \setminus \{ x_2 = 0 \} \to U
\]
by
\[
	\Phi_U(x_1,x_2) = (\frac {x_1}{x_2} , x_2^k) := (z,w),
\]
and similarly,
\[
 \Phi_{U'}: \mathbb C^2 \setminus \{ x_1 = 0 \} \to U'
\]
by
\[
	\Phi_{U'}(x_1,x_2) = (\frac {x_2}{x_1},x_1^k) := (z',w')
\]
which is seen to be well-defined and invariant under the $\mathbb Z_k$-action on $\mathbb C^2 \setminus \{ 0 \}$.

The covering has the property that
\[
	\Phi^*\rho = k \log r^2
\]
where $r = (|x_1|^2 + |x_2|^2)^{1/2}$.

Hence
\begin{align*}
	\Phi^* \widetilde \omega_T
	& = \Phi^* \Big( e^{\alpha \rho/2} ( k \ddbar \rho + \im \de \rho \wedge \db \rho ) \Big) \\
	& = k^2 r^{\alpha k} \big( \ddbar \log r^2 + \im \de \log r^2 \wedge \db \log r^2 \big) \\
	& = \frac {k^2} {r^{2 - \alpha k}} \im \sum_i {dx_i \wedge d \ov x_i}
\end{align*}
where we used that
\[
	\ddbar \log r^2 + \im \de \log r^2 \wedge \db \log r^2 = \im \frac 1 {r^2} \sum_i{ dx_i \wedge d \ov x_i}.
\]
Note that we have $0 < 2 - \alpha k < 2$ by condition ~\eqref{eq}. 

Thus
\[
 \Phi^* \omega_T \leq C \Phi^* \widetilde \omega_T = \frac C {r^{2 - \alpha k}} \omega_{\mathrm{eucl}},
\]
and therefore if $B_\delta(0)\subseteq \mathbb C^2$ is a Euclidean ball of radius $\delta>0$, then
\begin{equation}\label{eq:11}
	\mathrm{diam}(B_\delta(0)\setminus \{0\},\Phi^*\omega_T) < C \delta^{\alpha k/2}.
\end{equation}

Now, $\Phi^{-1}(W_{1/2} \setminus D_0)$ is identified with $B \setminus \{ 0 \}$, where $B \subseteq \mathbb C^2$ is the Euclidean unit ball, and from ~\eqref{eq:11} it follows that the metric completion of $(\Phi^{-1}(W_{1/2} \setminus D_0),\Phi^*\omega_T)$ is homeomorphic to $B$, and therefore the metric completion of $(W_{1/2}\setminus D_0,\omega_T)$ is homeomorphic to $B/\mathbb Z_k$ and $\ov X$ is homeomorphic to $\PP^2/\mathbb Z_k$.

Moreover, the diameter bound ~\eqref{eq:11} implies that
\begin{equation}\label{eq:10}
 \lim_{\delta \to 0} \limsup_{t \to T^-} \mathrm{diam} (W_\delta,d_{\omega(t)}) = 0,
\end{equation}
so that the exceptional curve is indeed contracting to a single point.

Using ~\eqref{eq:10} and that $\omega(t)$ converges to $\omega_T$ smoothly on compact subsets of $X \setminus D_0$ one can show $(X,\omega(t))$ converges to $(\ov X, d)$ in the Gromov-Hausdorff topology, and we have proved part (\textit{i}) of Theorem \ref{main}. \end{proof}

\section{Estimates Along the Flow with Symmetry in the Collapsing Case}

Next, we prove estimates on the evolving potential in the collapsing case. If $\alpha$ satisfies conditions ~\eqref{eq'} of Theorem \ref{main} then $(b_t - a_t) \to 0$ as $t \to T^-$, while $a_t > a_T > 0$ for all $0\leq t < T$. We obtain the following estimates on the potential.

\begin{lem}\label{lem3}
If $\alpha$ satisfies conditions ~\eqref{eq'}  of Theorem \ref{main}, then there is a uniform constant $C>0$, depending on $\psi$ and the initial K\"ahler class but independent of $\epsilon$, such that the following estimates hold:
 \begin{enumerate}[(i)]
  \item $0 < v_\epsilon'(t,\rho) - a_t < (1 + \alpha) (T - t)$
  \item $\lim_{t \to T^-} (v_\epsilon(t,\rho) - a_T \rho) = 0$
  \item $0 \leq v_\epsilon''(t,\rho) \leq C \mathrm{min} \big( \frac{e^{\alpha \rho}}{(1 + e^\rho)^{1+\alpha}},T - t \big)$
  \item $|v_\epsilon'''| \leq C v_\epsilon''$
 \end{enumerate}
\end{lem}

This lemma is similar to Lemma 4.1 and Lemma 4.3 in \cite{SoWe11b}. The proofs of parts (\textit{i}), (\textit{ii}), and (\textit{iii}) use similar methods to those presented in that paper. The proof of part (\textit{iv}), however, is complicated by the extra terms in our parabolic equation, and requires a different maximum principle argument.

\begin{proof}
Part (\textit{i}) follows from convexity of $v_\epsilon$ and the definition of $a_t$ and $b_t$.

Applying the bound in part (\textit{i}),
\begin{align*}
  |v_\epsilon(t,\rho) - a_t \rho| & = |\int_0^\rho { (v_\epsilon' - a_t) d\rho }| \\
  & \leq (1 + \alpha)(T - t) |\rho| \to 0
\end{align*}
as $t \to T$ while $a_t \to a_T$.

To prove the estimate
\[
	v_\epsilon''(\rho) \leq C \frac {e^{\alpha \rho}}{(1 + e^\rho)^{1+\alpha}},
\]
we use Lemma \ref{CKRF1} with $\widehat \omega = \ddbar \widehat v(\rho)$ the smooth K\"ahler metric defined in ~\eqref{eq:8} and ~\eqref{eq:8'}. We obtain
\[
	\omega_\epsilon(t)^2 = v_\epsilon' v_\epsilon'' \ddbar \rho \wedge \im \de \rho \wedge \db \rho \leq C \frac {\widehat v' \widehat v''}{|\sigma|_\eta^{2(1-\alpha)}} \ddbar \rho \wedge \im \de \rho \wedge \db \rho,
\]
and so
\[
	v_\epsilon' v_\epsilon'' \leq C \frac {\widehat v' \widehat v''}{|\sigma|_\eta^{2(1-\alpha)}} \leq C b (b-a) \frac {e^{\alpha \rho}}{(1 + e^\rho)^{1+\alpha}}.
\]
Using that $v_\epsilon'$ is uniformly bounded away from zero along the flow, we obtain the desired estimate.

To prove part (\textit{iv}) we compute the evolution of
\[
  Q = \frac {v_\epsilon'''}{v_\epsilon''} + (1-\alpha)(\log \theta)' - A t
\]
where $A$ is a large constant to be determined.

We claim that for each $t < T$ fixed, $Q$ is uniformly bounded from above and below independent of $\epsilon$ as $\rho$ tends to $\pm \infty$. By Lemma ~\ref{eq:9}, it suffices to bound $|v_\epsilon'''/v_\epsilon''|$ as $\rho \to \pm\infty$.

Since $\omega_\epsilon(t)$ remains smooth for all $t < T$, by Lemma \ref{calabi} there exist smooth functions $v_{0,\epsilon},v_{\infty,\epsilon} : [0,\infty) \times [0,T) \rightarrow \mathbb R$ with $v'_{0,\epsilon}(0,t)>0$ and $v'_{\infty,\epsilon}(0,t)>0$ such that
\begin{align*}
	v_\epsilon(\rho,t) 
	& = v_{0,\epsilon}(e^\rho,t) + a_t \rho \\
	& = v_{\infty,\epsilon}(e^{-\rho},t) + b_t \rho.
\end{align*}
So that as $\rho$ tends to $-\infty$,
\[
 \frac {v_\epsilon'''}{v_\epsilon''} = \frac {v_{0,\epsilon}''' e^{2 \rho} + 3 v_{0,\epsilon}'' e^\rho + v_{0,\epsilon}'} {v_{0,\epsilon}'' e^\rho + v_{0,\epsilon}'} \to 1.
\]
The bound as $\rho$ tends to $+\infty$ is similar.

Now, fix a time $0< T' < T$. If $Q$ achieves a maximum at $(x_0,t_0) \in \mathbb R \times (0,T']$, then at $(x_0,t_0)$,
\[
	0 = Q' = \frac {v_\epsilon^{(4)}}{v_\epsilon''} - \frac {(v_\epsilon''')^2}{(v_\epsilon'')^2} + (1-\alpha)(\log \theta)'',
\]
and
\[
 0 \geq Q'' = \frac {v_\epsilon^{(5)}}{v_\epsilon''} - 3 \frac {v_\epsilon''' v_\epsilon^{(4)}}{(v_\epsilon'')^2} + 2 \frac {(v_\epsilon''')^2}{(v_\epsilon'')^2} + (1-\alpha)(\log \theta)'''.
\]
By the parabolic maximum principle, at $(x_0,t_0)$ we have 
\begin{align*}
	0 \leq v_\epsilon'' \frac {\de Q}{\de t} 
	= & \big( \frac {\de v_\epsilon'''}{\de t} - \frac{v_\epsilon'''}{v_\epsilon''} \frac{\de v_\epsilon''}{\de t} - A v_\epsilon'' \big) \\
	= & \frac {v_\epsilon^{(4)}} {v_\epsilon'} - 3 \frac {v_\epsilon'' v_\epsilon'''}{(v_\epsilon')^2} + 2 \frac {(v_\epsilon'')^3} {(v_\epsilon')^3} + \big( \frac {v_\epsilon^{(5)}} {v_\epsilon''} - 3 \frac {v_\epsilon'''} {v_\epsilon''} \frac {v_\epsilon^{(4)}} {v_\epsilon''} + 2 \frac {(v_\epsilon''')^3} {(v_\epsilon'')^3} + (1-\alpha)(\log \theta)''' \big) \\
	& - \frac {v_\epsilon'''}{v_\epsilon''} \Big( -  \frac {(v_\epsilon'')^2}{(v_\epsilon')^2} + \frac {v_\epsilon'''} {v_\epsilon'} + \big( \frac{v_\epsilon^{(4)}} {v_\epsilon''} - \frac {(v_\epsilon''')^2} {(v_\epsilon'')^2} + (1-\alpha)(\log \theta)'' \big) \Big) - A v_\epsilon''\\
	= & 2 \big( \frac{v_\epsilon''}{v_\epsilon'} \big)^3 - 2 \big( \frac{v_\epsilon''}{v_\epsilon'} \big)^2 \big( Q - (1-\alpha) (\log \theta)' + At \big) \\
	& + \frac {v_\epsilon''} {v_\epsilon'} \big( Q' - (1-\alpha) (\log \theta)'' \big) - \frac {v_\epsilon'''}{v_\epsilon''} Q' + Q''  - A v_\epsilon'' \\
	\leq & 2 \big( \frac{v_\epsilon''}{v_\epsilon'} \big)^3 - 2 \big( \frac{v_\epsilon''}{v_\epsilon'} \big)^2 \big( Q - (1-\alpha) (\log \theta)' + At \big) - (1-\alpha) \frac {v_\epsilon''}{v_\epsilon'} (\log \theta)''  - A v_\epsilon'' \\
	\leq & 2 \big( \frac{v_\epsilon''}{v_\epsilon'} \big)^3 - 2 \big( \frac{v_\epsilon''}{v_\epsilon'} \big)^2 \big( Q - (1-\alpha) (\log \theta)' + At \big) + C v_\epsilon''  - A v_\epsilon''
\end{align*}
where we used Lemma ~\ref{eq:9} and that $v_\epsilon'$ is uniformly bounded away from zero in the last line. Choosing $A = C+1$, and using that $v_\epsilon'' > 0$, we conclude
\[
  Q(x_0,t_0) \leq \big( \frac {v_\epsilon''}{v_\epsilon'} + (1-\alpha) (\log \theta)' - A t \big) (x_0,t_0) \leq C,
\]
and therefore on $\mathbb R \times [0,T']$
\[
 Q = \frac{v_\epsilon'''}{v_\epsilon''} + (\log \theta)' - A t \leq C.
\]
Using Lemma ~\ref{eq:9} again,
\[
  v_\epsilon''' \leq (C + A T) v_\epsilon'',
\]
and since the bound is independent of $T'<T$, letting $T'$ tend to $T$ we obtain the estimate on all of $\mathbb R \times [0,T)$.

The proof of the lower bound can be done similarly by computing the evolution of
\[
 Q = \frac{v_\epsilon'''}{v_\epsilon''} + (1-\alpha)(\log \theta)' + A t
\]
at a point of minimum.

Finally, to prove
\[
	v_\epsilon'' \leq C (T - t),
\]
we use that $v_\epsilon''(t,\rho) \leq C \widehat v''(\rho)$ implies that for each $t\in [0,\infty)$ fixed, $v_\epsilon''(t,\rho) \to 0$ as $\rho \to \pm \infty$. 
Thus there exists $\rho_t \in \mathbb R$ such that
\[
 v_\epsilon''(t,\rho_t) = \sup_{\rho \in \mathbb R} v_\epsilon''(t,\rho).
\]
Next, the Mean Value Theorem and the bound in part (\textit{iv}) imply
\[
 v_\epsilon''(t,\rho_t) - v_\epsilon''(t,\rho) \leq C v_\epsilon''(t,\rho_t) | \rho - \rho_t|
\]
where the constant is independent of $\epsilon$, so that for $|\rho - \rho_t| \leq 1/2C$,
\[
 v_\epsilon''(t,\rho) \geq \frac {v_\epsilon''(t,\rho_t)} 2,
\]
and so
\[
 \frac 1 {2C} v_\epsilon''(t,\rho_t) = \int_{|\rho - \rho_t|\leq 1/2C} { \frac{v_\epsilon''(t,\rho_t)} 2 d\rho} < \int_{-\infty}^{\infty} {v_\epsilon''(t,\rho) d\rho} = b_t - a_t = (1 + \alpha) (T-t)
\]
from which the bound follows.
\end{proof}

From the previous lemma we obtain the following immediate corollaries.
\begin{cor}\label{cor4}
There exists a uniform constant $C>0$, independent of $\epsilon$, such that for all $t \in [0,T)$
  \begin{enumerate}[(i)]
    \item $\omega_\epsilon(t) \leq C (\widehat v'(\rho) \ddbar \rho + |\sigma|_\eta^{-2 (1-\alpha)} \widehat v''(\rho) \im \de \rho \wedge \db \rho)$
    \item $(ka_T) p^*\omega_{\mathrm{fs}} \leq \omega_\epsilon(t)$
    \item $C^{-1} \leq \mathrm{diam}(X,\omega_\epsilon(t)) \leq C$
  \end{enumerate}
\end{cor}

\begin{proof}
 Part (\textit{i}) follows from the calculation that
\[
 \frac 1 {|\sigma|_\eta^{2(1-\alpha)}} \widehat v''(\rho) = \frac {e^{\alpha \rho}}{(1+e^\rho)^{1 + \alpha}}
\]
and the first estimate of Lemma ~\ref{lem3}.\textit{iii}; part (\textit{ii}) follows from Lemma ~\ref{lem3}.\textit{i}; and part (\textit{iii}) follows from part (\textit{ii}) for the lower bound, and part (\textit{i}) for the upper bound using the fact that
  \[
   \omega_{\mathrm{cone}} = \widehat v'(\rho) \ddbar \rho +  \frac {\widehat v''}{|\sigma|_\eta^{2 (1-\alpha)}} \im \de \rho \wedge \db \rho
  \]
defines a conical K\"ahler metric on $X$ with cone angle $2\pi \alpha$ along $D_0$, such that the underlying metric space has finite diameter.
\end{proof}

We claim that the fibers are uniformly shrinking to a point for any solution to the twisted K\"ahler-Ricci flow ~\eqref{eq:2}.
\begin{lem}\label{lem5}
 Let $F_y = p^{-1}(y)$ for some $y \in \PP^1$, then
\[
 \lim_{t \to T^-} \mathrm{diam}(F_y,\omega_\epsilon(t)) = 0.
\]
Moreover, the convergence is uniform in the sense that for any $\delta>0$ there exists a constant $C>0$, independent of $\epsilon>0$, such that 
\begin{equation}\label{eq:13}
 \sup_{y \in \PP^1} \mathrm{diam}(F_y,\omega_\epsilon(t)) < \delta
\end{equation}
for all $0 < T - t < C^{-1}$.
\end{lem}

\begin{proof}
Let $\delta>0$ be given. From Corollary ~\ref{cor4}.\textit{i} there exist open neighborhoods $W^0$ and $W^\infty$ containing $D_0$ and $D_\infty$, respectively, such that
\[
 \mathrm{diam}(W^0 \cap F_y,\omega_\epsilon(t)) < \frac \delta 4,
\]
and
\[
 \mathrm{diam}(W^\infty \cap F_y,\omega_\epsilon(t)) < \frac \delta 4
\]
for all $t \in[0,T)$ and every $y \in \PP^1$.

Let $K = X \setminus (W^0 \cup W^\infty) \subset \subset X \setminus (D_0 \cup D_\infty)$. From the second estimate of Lemma ~\ref{lem3}.\textit{iii}, there exists a constant $C_K>0$, depending on $K$ but independent of $\epsilon$, such that
\[
  \sup_{y \in \PP^1} \|\omega_\epsilon(t)|_{F_y} \|_{C^0(F_y \cap K)} \leq C^2_K (T -t),
\] 
and so we have 
\[
 \sup_{y \in \PP^1} \mathrm{diam}(K \cap F_y, \omega_\epsilon(t)) \leq C_K (T - t)^{1/2}
\]
for $C_K$ possibly larger but depending only on $K$.

It follows that
\[
 \sup_{y \in \PP^1} \mathrm{diam}(F_y,\omega_\epsilon(t)) < C_K (T-t)^{1/2} + \frac \delta 2,
\]
and so for $t$ sufficiently close to $T$ we have ~\eqref{eq:13} uniformly and letting $\delta \to 0$ we obtain
\[
	\lim_{t \to T^-} \mathrm{diam}(F_y,\omega_\epsilon(t)) = 0
\]
for every $y \in \PP^1$, and we have proved the lemma.
\end{proof}

We are now in a position to compute the Gromov-Hausdorff limit for the twisted K\"ahler-Ricci flows at the singularity time.
\begin{lem}\label{lem6}
 $(X,\omega_\epsilon(t))$ Gromov-Hausdorff converges to $(\PP^1,(ka_T) \omega_{\mathrm{fs}})$ as $t \to T^-$.
\end{lem}
\begin{proof}
 Let
\[
 p: X \to \PP^1
\]
be the map giving $X$ the structure of a ruled surface, as in ~\eqref{eq:15}, and let
\[
 s: \PP^1 \to X
\]
be a holomorphic section of the projective line bundle satisfying $p \circ s = \mathrm{Id}_{\PP^1}$ and giving an isomorphism of $\PP^1$ onto $D_\infty \subseteq X$.

Let $\delta>0$ be given. We claim that $p$ and $s$ are $\delta$-isometries of $(X,\omega_\epsilon(t))$ and $(\PP^1,(ka_T)\omega_{\mathrm{fs}})$ for $t$ sufficiently close to $T$. 

First, by Corollary ~\ref{cor4}.\textit{ii} and Lemma ~\ref{lem3} we have
\[
  (k a_T) \omega_{\textrm{fs}} \leq s^* \omega_\epsilon(t) \leq k (a_T + (1+\alpha)(T - t)) \omega_{\textrm{fs}}.
\]
Hence for any $y_0,y_1 \in \PP^1$,
\begin{equation}\label{eq:16}
 (k a_T)^{1/2} d_{\omega_{\mathrm{fs}}}(y_0,y_1) \leq d_{\omega_\epsilon(t)}(s(y_0),s(y_1)) \leq (k a_T + k (1+\alpha)(T-t))^{1/2} d_{\omega_{\mathrm{fs}}}(y_0,y_1).
\end{equation}
In particular,
\[
 |d_{\omega_\epsilon(t)}(s(y_0),s(y_1)) - (k a_T)^{1/2}d_{\omega_{\mathrm{fs}}}(y_0,y_1)| \leq k^{1/2}(1+\alpha)^{1/2}(T-t)^{1/2} \mathrm{diam}(\PP^1,\omega_{\mathrm{fs}}) < \delta
\]
for $t$ sufficiently close to $T$.

Next, there exists $C>0$ such that
\[
  \mathrm{diam}(F_y,\omega_\epsilon(t)) < \frac \delta 4
\]
for all $0 < T - t < C^{-1}$ and all $y \in \PP^1$. Thus for all $x \in X$ we have
\[
 d_{\omega_\epsilon(t)}(x,(s \circ p)(x)) < \frac \delta 4,
\]
and therefore for all $x_0,x_1 \in X$, using ~\eqref{eq:16},
\begin{align*}
  (k a_T)^{1/2} & d_{\omega_{\textrm{fs}}}( p(x_0),p(x_1))\\
  & \leq d_{\omega_\epsilon(t)}(x_0,x_1) \\
  & \leq d_{\omega_\epsilon(t)}(x_0,(s \circ p)(x_0)) + d_{\omega_\epsilon(t)}((s \circ p)(x_0),(s \circ p)(x_1)) + d_{\omega_\epsilon(t)}((s \circ p))(x_1),x_1) \\
  & \leq (k a_T)^{1/2} d_{\omega_{\mathrm{fs}}}(p(x_0),p(x_1)) + k^{1/2} (1+\alpha)^{1/2} (T - t)^{1/2} \mathrm{diam}(\PP^1,\omega_{\mathrm{fs}}) + \frac \delta 2 \\
  & \leq (k a_T)^{1/2} d_{\omega_{\mathrm{fs}}}(p(x_0),p(x_1)) + \delta
\end{align*}
for $t$ sufficiently close to $T$. Thus
\[
 \big| d_{\omega_\epsilon(t)}(x_0,x_1) - (k a_T)^{1/2} d_{\omega_{\mathrm{fs}}}(p(x_0),p(x_1)) \big| \leq \delta
\]
for all $x_0,x_1 \in X$ and $t$ sufficiently close to $T$, which proves the claim.
\end{proof}

Because the estimates are independent of $\epsilon>0$ we obtain the following estimates for $\omega(t)$ solving the conical K\"ahler-Ricci flow, and complete the proof of Theorem \ref{main}.\textit{ii}.
\begin{cor}
 There exists a uniform constant $C>0$, such that for all $t \in [0,T)$:
  \begin{enumerate}[(i)]
    \item $\omega(t) \leq C (\widehat v'(\rho) \ddbar \rho + |\sigma|_\eta^{-2 (1-\alpha)} \widehat v''(\rho) \im \de \rho \wedge \db \rho)$
    \item $(k a_T) p^*\omega_{\mathrm{fs}} \leq \omega(t)$
    \item $C^{-1} \leq \mathrm{diam}(X,\omega(t)) \leq C$
    \item $\lim_{t \to T^-} \mathrm{diam}(F_y,\omega(t)) = 0$ for all $y \in \PP^1$ and the convergence is uniform.
    \item $(X,\omega(t))$ converges to $(\PP^1,(ka_T)\omega_{\mathrm{fs}})$ in Gromov-Hausdorff topology as $t \to T^-$.
  \end{enumerate}
\end{cor}
\begin{proof}
 Let $\epsilon_i \to 0^+$ be a subsequence such that $\omega_{\epsilon_i}(t)$ converges to $\omega(t)$, the solution to the conical K\"ahler-Ricci flow, weakly on $X \times [0,T)$ and in $C^\infty_{\mathrm{loc}}((X\setminus D_0) \times [0,T))$. The estimates (\textit{i}) and (\textit{ii}) of Corollary ~\ref{cor4} are independent of $\epsilon$, so we obtain estimates (\textit{i}) and (\textit{ii}) on the solution to the conical K\"ahler-Ricci flow. Part (\textit{iii}) then follows from estimates (\textit{i}) and (\textit{ii}). Part (\textit{iv}) follows by the same proof as in Lemma ~\ref{lem5}, and part (\textit{v}) follows by the same argument as in Lemma ~\ref{lem6} using the fact that $\omega(t)$ remains smooth across $D_\infty$ for all $t \in [0,T)$.
\end{proof}

\section{Further conjectures}

We now illustrate a conjectural picture to describe the types of finite time non-collapsing singularities we expect for the conical K\"ahler-Ricci flow on surfaces.

Let $X$ be a smooth compact K\"ahler surface, and $D = \sum_i \beta_i D_i$ an $\mathbb R$-divisor with coefficients $0 < \beta_i < 1$, with $D_i$ smooth irreducible divisors, and with $D$ having simple normal crossing support. Let $\omega(t)$ be a solution of the conical K\"ahler-Ricci flow starting with an initial conical K\"ahler metric on $(X,D)$, that is with cone angles $2 \pi (1-\beta_i)$ along $D_i$ for each $i$ (see e.g. \cite{Ru14,Ti96} for a definition of conic metrics with cone angles along a divisor with simple normal crossing support). Suppose $\omega(0)\in[\omega_0]$, and assume $\omega(t)$ reaches a finite time non-collapsing singularity at $T < \infty$.

We conjecture that at the singular time, the conical K\"ahler-Ricci flow must contract some collection of disjoint curves $E_1,...,E_\ell$ which are non-singular and isomorphic to $\PP^1$. For such curves we must have
\[
  \mathrm{Vol}(E_i,\omega(t)) = E_i \cdot [\omega(t)] \to 0.
\]
\textit{Note: the volume of $E_i$ may not be well-defined if $E_i$ is an irreducible component of the cone divisor $D$, in this case the volume can then be defined using cohomology}.

In terms of the intersection ring, the previous condition implies the necessary condition:
\[
  (K_X + D) \cdot E_i < 0,
\]
where $K_X$ is the canonical bundle of $X$. Moreover $[\omega_0] + T(K_X + D)$ is a big and nef class such that
\[
	\big( [\omega_0] + T(K_X + D) \big) \cdot E_i = 0.
\]
By the Hodge Index Theorem \cite{BHPV04} it follows that $E_i^2 \leq -1$ for each $i$ and
\[
 0 > (E_i + E_j)^2 = 2 E_i \cdot E_j + E_i^2 + E_j^2
\]
for any $i$ and $j$. In particular, $E_i \cdot E_j = 0$ for $i \neq j$, and thus the $E_i$ are disjoint.

Now, recall the adjunction formula on surfaces for smooth embedded curves from algebraic geometry:
\[
  2 g(E) - 2 = (K_X + E) \cdot E. 
\]
where $g(E)$ is the geometric genus of $E$.

The curves must fall into one of the three following types:
\begin{enumerate}[(i)]
  \item \emph{$E$ is disjoint from the support of $D$}. In this case $E \cdot D$ = 0, and thus $K_X \cdot E < 0$, and by adjunction, $g(E) = 0$, $E^2 = -1$, and $E \cong \PP^1$. In other words, $E$ is a $(-1)$-curve.
  \item \emph{$E$ is an irreducible component of the support of $D$}. Write $D = D' + \beta E$, where $D'$ is an effective divisor without $E$ as an irreducible component. Note that we have $D' \cdot E \geq 0$, and
  	\[
  		(K_X + D' + \beta E) \cdot E < 0.
  	\]
  By adjunction,
  	\begin{align*}
  		2 g(E) - 2 & = (K_X +E) \cdot E \\
  		& < - D' \cdot E + (1-\beta)E^2 \leq (1-\beta) E^2 < 0.
  	\end{align*}
  Thus $g(E) = 0$, so $E \cong \PP^1$, and therefore
  	\[
  		\frac {-2 + D' \cdot E}{1-\beta} < E^2 \leq -1
  	\]
  which is only possible if $D' \cdot E < 1 + \beta$, or stated another way: if $\alpha = (1-\beta)$ is the cone angle$/2 \pi$ along $E$, then
 	 \[
  		\alpha < \frac {2 - D' \cdot E} {(-E^2)}.
 	 \]
 	 In particular, if $E$ is disjoint from all other irreducible components of $D$, then
 	 \[
 	 		\alpha < \frac 2 {(-E^2)}.
 	 \]
	\item \emph{$E$ intersects the support of $D$, but is not an irreducible component of $D$}. Since $E$ is distinct from the irreducible components of $D$,
	\[
		E \cdot D \geq 0, 
	\]
	and since
	\[
  	(K_X + D) \cdot E < 0
	\]
	the adjunction formula yields
	\[
		2 g(E) - 2 = (K_X + E) \cdot E < - D \cdot E + E^2 \leq -1.
	\]
	It follows that $g(E) = 0$, so $E \cong \PP^1$, and $0 \leq D \cdot E < E^2 + 2 \leq 1$, which is only possible if $E^2 = -1$ and $0 \leq D \cdot E < 1$.
\end{enumerate}

By Theorem \ref{main}, we have provided an explicit example of a contraction of type (\textit{ii}) with $D = (1 - \alpha)D_0$.

Moreover, this suggests that given a curve $E$ having negative self-intersection, then for sufficiently small cone angle along $E$, the conical K\"ahler-Ricci flow may contract $E$ at the singular time for some choices of initial conical K\"ahler metrics.

We claim that contractions of \emph{singular} embedded curves do not happen in non-collapsing singularities.
\begin{prop}
If $\omega(t)$ is a maximal solution of the conical K\"ahler-Ricci flow on a K\"ahler surface $X$ such that
\[
 \mathrm{Vol}(X,\omega(t)) > \lambda > 0
\]
for all $t \in [0,T)$, and $E$ is a compact irreducible codimension one subvariety such that
\[
	\mathrm{Vol}(E,\omega(t)) \to 0
\]
as $t \to T^-$, then $E$ is non-singular and isomorphic to $\PP^1$.
\end{prop}

\begin{proof}
Indeed if $E$ is presumed to be singular, then the adjunction formula on surfaces takes the form
\[
	2 p_a(E) - 2 = (K_X + E) \cdot E
\]
where $p_a(E) = \mathrm{dim}_{\mathbb C} H^1(E,\mathcal O_E)$ is the arithmetic genus of $E$.

Now, by assumption $E$ satisfies
\[
 \big( [\omega_0] + T(K_X + D) \big) \cdot E = 0,
\]
so $E^2 \leq -1$ by the Hodge Index Theorem. Moreover, $E$ satisfies
\[
 (K_X + D) \cdot E < 0,
\]
and $E$ falls into one of the three cases outlined above. In each case we conclude that $p_a(E) = 0$, from which it follows that $E \cong \PP^1$ and is therefore non-singular \cite{GH78}.
\end{proof}

We wish to outline this conjectural picture in more detail. Let $X$ be a smooth compact K\"ahler surface, and $D = \sum_i \beta_i D_i$ a cone divisor with simple normal crossing support and $0< \beta_i<1$. We may abuse notion by writing $D$ for both the divisor itself and for its support, where context is clear. 

Now, we begin the conical K\"ahler-Ricci flow with an initial conic metric on $(X,D)$. Suppose that the conical K\"ahler-Ricci flow on $(X,D)$ reaches a finite time non-collapsing singularity contracting some collection of curves $E_1,...,E_\ell$. Write $E = \bigcup_i E_i$. Then there is a non-negative current $\omega_T$ such that as $t \to T^-$, $\omega(t)$  converges to $\omega_T$ globally in the sense of currents and smoothly on compact subsets of $X \setminus (D \cup E)$. Let $(\ov X, d)$ be the metric completion of $(X \setminus (D \cup E),\omega_T)$. We conjecture that $(X,\omega(t))$ converges to $(\ov X, d)$ in the Gromov-Hausdorff topology. Furthermore, we conjecture that $\ov X$ is homeomorphic to a projective variety $Y$, possibly with mild singularities, and that there exists a birational morphism
\[
  f: X \rightarrow Y
\]
such that $f$ contracts each $E_i$ to a point, say $f(E_i) = y_i \in Y$, and $f$ is a biholomorphism from $X \setminus E$ to $Y \setminus \{ y_1,...,y_\ell \}$. Define $D' = f_* D$ as an $\mathbb R$-divisor on $Y$.

This morphism should correspond precisely to the extremal contraction prescribed by the \emph{log minimal model program with scaling of} $[\omega_0]$ (see \cite{BiCaHaMc06}). In dimension two, these are always divisorial contractions, but in higher dimensions $f$ may only be a birational transformation such as a flip. If the map comes from the log MMP with scaling, then assuming $X$ is smooth and $D$ has simple normal crossing support, $(Y,D')$ will have at worst log terminal singularities. In dimension two, this means $Y$ can only have isolated orbifold singularities.

However, $D'$ may no longer have simple normal crossing support, and in fact, its irreducible components may no longer even be smooth in general. If $D'$ \emph{does} have simple normal crossing support, we expect that $f_* \omega_T$ defines a conical K\"ahler metric on $Y \setminus \{ y_1,...,y_\ell \}$  with cone divisor $D'$ in the usual sense. That is, $f_* \omega_T$ is a positive current which is a smooth K\"ahler metric on $Y \setminus ( D' \cup \{ y_1,...,y_\ell \})$, and is quasi-isometric to the standard cone metric along $D'$ away from the points $\{ y_1,...,y_\ell \}$.

If $D'$ has simple normal crossing support and $f_* \omega_T$ defines a conical metric on $(Y\setminus \{ y_1,...,y_\ell \},D')$ in the sense described above, one can hope to continue the conical K\"ahler-Ricci flow on $(Y,D')$ in a weak sense, such as in \cite{LiZh16}.

If $D'$ does not have simple normal crossing support, we still expect $f_* \omega_T$ to be asymptotic to the standard cone metric outside of finitely many points in $D'$ which correspond to the singular points on the irreducible components of $D'$ and the intersection points which are not simple normal crossing. In this case one would hope to continue the conical K\"ahler-Ricci flow in a weak sense on $(Y,D')$. The existence of a solution to such a parabolic flow has not yet been established in sufficient generality.

\section{Further Examples}

We now present a specific example where we expect to see a finite time non-collapsing singularity of type (\textit{iii}) such that the Gromov-Hausdorff limit is homeomorphic to a smooth projective variety with metric singularities along a divisor without simple normal crossing support, illustrating the phenomena conjectured above.

Specifically, let $X$ be a Hirzebruch surface of degree $k$ with
\[
	p: X \to \PP^1
\]
the ruled surface map, and
\[
	\pi: X \to \PP^2 / \mathbb Z_k
\]
the blow-up map centered at the orbifold point $y_0 \in \PP^2 / \mathbb Z_k$.

Let $w_1,...,w_\ell \in \PP^1$ be a finite number of distinct points and $F_i = p^{-1}(w_i)$ be the fibers over each point. Define
\[
	D = \sum_{i=1}^\ell \beta_i F_i
\]
to be the cone divisor with $0<\beta_i<1$, so that $2 \pi (1 - \beta_i)$ is the cone angle along $F_i$, and let 
\[
	\beta = \sum_{i=1}^\ell \beta_i.
\]

\begin{conj}
Let $\omega_0$ be a conical K\"ahler metric on $(X,D)$ defined in this way in the class $2 \pi (b[D_\infty] - a [D_0])$ with $0 < a < b$. 
\begin{enumerate}[(i)]
	\item If $k=1$, and
	\[
		\frac {2 a}{b-a} < (1-\beta)
	\]
	then the conical K\"ahler-Ricci flow starting with $\omega_0$ reaches a finite time non-collapsing singularity at time $T = \frac {a}{1 - \beta}$ which contracts the zero section, $D_0$, at the singularity time (see Figure \ref{fig}).
	
	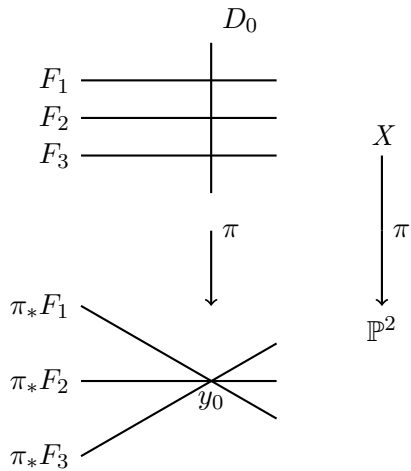
\begin{figure}
\centering
\begin{tikzpicture}
  \draw[thick] (0,0) node[anchor=east] {$\pi_* F_3$} -- (2.60,1.5);
  \draw[thick] (0,1) node[anchor=east] {$\pi_* F_2$} -- (1.73,1.0) node[anchor=north] {$y_0$};
  \draw[thick] (1.73,1) -- (2.60,1.0) ;
  \draw[thick] (0,2) node[anchor=east] {$\pi_* F_1$} -- (2.60,0.5);
  \draw[thick,->] (1.73,3) node[anchor=west] {$\pi$} -- (1.73,2);
  \draw[thick] (1.73,3.5) -- (1.73,5.5) node[anchor=south west] {$D_0$};
  \draw[thick]  (2.60,4) -- (0,4) node[anchor=east] {$F_3$};
  \draw[thick]  (2.60,4.5) -- (0,4.5) node[anchor=east] {$F_2$};
  \draw[thick]  (2.60,5) -- (0,5) node[anchor=east] {$F_1$};
  \draw[thick] (4,4) node[anchor=south] {X} -- (4,3) node[anchor=west] {$\pi$};
  \draw[thick,->] (4,3) -- (4,2) node[anchor=north] {$\mathbb{P}^2$};
\end{tikzpicture}
\caption{The arrangement of hypersurfaces and the expected contraction of Conjecture 9.1.\textit{i} with $k = 1$ and $\ell =3$.} \label{fig}
\end{figure}

	Furthermore, as $t \to T^-$, $\omega(t) \to \omega_T$ in the sense of currents, with $\omega_T$ a smooth K\"ahler metric on $X \setminus (D \cup D_0)$. If $(\ov X, d_T)$ is the metric completion of $(X \setminus (D \cup D_0), \omega_T)$, then $\ov X$ is homeomorphic to $\PP^2$,
	\[
	  \pi: X \to \PP^2
	\]
	is the divisorial contraction outlined above, and $(X,\omega(t))$ Gromov-Hausdorff converges to $(\ov X,d_T)$ as $t \to T^-$.

	Define
	\[
		D' = \pi_* D = \sum_i^\ell \beta_i \pi_*F_i.
	\]
	Then $D'$ defines a divisor on $\PP^2$ whose support consists of a finite number of hyperplanes passing through the blow-up center and in particular, if $\ell \geq 3$, then $D'$ cannot have simple normal crossing support.

	Moreover, $\pi_* \omega_T$ defines a conical K\"ahler metric on $(\PP^2\setminus \{ y_0 \}$, $D' \setminus \{y_0 \})$ in the sense that $\pi_* \omega_T$ is a smooth K\"ahler metric on $Y \setminus D'$, and for all $x \in \pi(F_i)$, $x \neq y_0$, $\pi_* \omega_T$ is quasi-isometric to
	\[
		\im \big( \frac {dz^1 \wedge d \ov z^1}{|z^1|^{2\beta_i}} + dz^2 \wedge d \ov z^2 \big)
	\]
	in a coordinate patch $U \subset \subset \PP^2 \setminus \{y\}$ centered at $x$ with coordinates $(z^1,z^2)$ such that $\pi(F_i) \cap U = \{ z^1 = 0 \}$.

\item If $k=1$, $\beta \in (0,1)$, and $\frac {2a}{b-a} = (1 - \beta)$, then the initial K\"ahler class is a positive multiple of $[K_X^{-1}] - [D]$, and $(X,\omega(t))$ Gromov-Hausdorff converges to a single point at the singularity time, as proved by Liu-Zhang \cite{LiZh}.

\item In all other cases, in particular whenever $k \geq 2$, we have 
	\[
		\mathrm{Vol}(X,\omega(t)) \to 0
	\]
	and $(X,\omega(t))$ Gromov-Hausdorff converges to a conical K\"ahler metric on $\PP^1$ with cone angle $2 \pi (1-\beta_i)$ at $w_i$ for each $i$.

\end{enumerate}
\end{conj}

\bibliographystyle{plain}
\bibliography{Bibliography}

\begin{thebibliography}{10}

\bibitem{BHPV04}
W.~Barth, K.~Hulek, C.~Peters, and A.~van~de Ven.
\newblock {\em Compact Complex Surfaces}, volume~4 of {\em Ergebnisse der
  Mathematik und ihrer Grenzgebiete. 3. Folge / A Series of Modern Surveys in
  Mathematics}.
\newblock Springer-Verlag, Berlin Heidelberg, 2 edition, 2004.

\bibitem{BiCaHaMc06}
C.~Birkar, P.~Cascini, C.~Hacon, and J.~McKernan.
\newblock Existence of minimal models for varieties of log general type.
\newblock {\em J. Amer. Math. Soc.}, 23(2):405--468, 2010.

\bibitem{brendle}
S.~Brendle.
\newblock Ricci flat {K\"ahler} metrics with edge singularities.
\newblock {\em Int. Math. Res. Not.}, 24:5727--5766, 2013.

\bibitem{Ca82}
E.~Calabi.
\newblock Extremal {K\"ahler} metrics.
\newblock In {\em Seminar on Differential Geometry}, volume 102 of {\em Ann. of
  Math. Stud.}, pages 259--290. Princeton Univ. Press, Princeton, N.J., 1982.

\bibitem{CGP}
F.~Campana, H.~Guenancia, and M.~P\u{a}un.
\newblock Metrics with cone singularities along normal crossing divisors and
  holomorphic tensor fields.
\newblock {\em Ann. Sc. \'Ec. Norm. Sup\'er. (4)}, 46(6):879--916, 2013.

\bibitem{cao}
H.D. Cao.
\newblock Deformation of {K\"ahler} metrics to {K\"ahler-Einstein} metrics on
  compact {K\"ahler} manifolds.
\newblock {\em Invent. Math}, 81(2):359--372, 1985.

\bibitem{CDS1}
X.X. Chen, S.~Donaldson, and S.~Sun.
\newblock {K\"ahler-Einstein} metrics on {Fano} manifolds, {I}: Approximation
  of metrics with cone singularities.
\newblock {\em J. Amer. Math. Soc.}, 28(1):183--197, 2015.

\bibitem{CDS2}
X.X. Chen, S.~Donaldson, and S.~Sun.
\newblock {K\"ahler-Einstein} metrics on {Fano} manifolds, {II}: Limits with
  cone andgle less than $2\pi$.
\newblock {\em J. Amer. Math. Soc.}, 28(1):199--234, 2015.

\bibitem{CDS3}
X.X. Chen, S.~Donaldson, and S.~Sun.
\newblock {K\"ahler-Einstein} metrics on {Fano} manifolds, {III}: Limits as
  cone angle approaches $2\pi$ and completion of the main proof.
\newblock {\em J. Amer. Math. Soc.}, 28(1):235--278, 2015.

\bibitem{ChenWang1}
X.X. Chen and Y.Q. Wang.
\newblock Bessel functions, heat kernel and the conical {K\"ahler-Ricci} flow.
\newblock {\em J. Funct. Anal.}, 269(2):551--632, \newsort{2013}2015.

\bibitem{ChenWang2}
X.X. Chen and Y.Q. Wang.
\newblock On the long time behaviour of the conical {K\"ahler-Ricci} flows.
\newblock arXiv: 1402.6689, \newsort{2014}2014.

\bibitem{CoTo15}
T.~Collins and V.~Tosatti.
\newblock K\"ahler currents and null loci.
\newblock {\em Invent. Math.}, 202(3):1167--1198, 2015.

\bibitem{DaSo15}
V.~Datar and J.~Song.
\newblock A remark on {K\"ahler} metrics with conical singularities along a
  simple normal crossing divisor.
\newblock {\em Bull. Lond. Math. Soc.}, 47(6):1010--1013, 2015.

\bibitem{Do11}
S.~Donaldson.
\newblock K\"ahler metrics with cone singularities along a divisor.
\newblock In {\em Essays in mathematics and its applications}, pages 49--79.
  Springer, Heidelberg, 2012.

\bibitem{E}
G.~Edwards.
\newblock A scalar curvature bound along the conical {K\"ahler-Ricci} flow.
\newblock {\em J. Geom. Anal.}, 2017.
\newblock DOI:10.1007/s12220-017-9817-0.

\bibitem{FIK}
M.~Feldman, T.~Ilmanen, and D.~Knopf.
\newblock Rotationally symmetric shrinking and expanding gradient
  {K\"ahler-Ricci} solitons.
\newblock {\em J. Differential Geom.}, 65(2):169--209, 2003.

\bibitem{GH78}
P.~Griffiths and J.~Harris.
\newblock {\em Principles of Algebraic Geometry}.
\newblock Wiley, 1978.

\bibitem{GP}
H.~Guenancia and M.~P\u{a}un.
\newblock Conic singularities metrics with prescribed {Ricci} curvature:
  general cone angles along normal crossing divisors.
\newblock {\em J. Differential Geom.}, 103(1):15--57, 2016.

\bibitem{GuoSo16}
B.~Guo and J.~Song.
\newblock Schauder estimates for equations with cone metrics, {I}.
\newblock arXiv:1612.00075, 2016.

\bibitem{GuSoWe16}
B.~Guo, J.~Song, and B.~Weinkove.
\newblock Geometric convergence of the {K\"ahler-Ricci} flow on surfaces of
  general type.
\newblock {\em Int. Math. Res. Notices}, 2016(18):5652--5669, 2016.

\bibitem{H82}
R.~Hamilton.
\newblock Three-manifolds with positive {Ricci} curvature.
\newblock {\em J. Differential Geom.}, 17(2):255--306, 1982.

\bibitem{JMR}
T.D. Jeffres, R.~Mazzeo, and Y.A. Rubinstein.
\newblock {K\"ahler-Einstein} metrics with edge singularities.
\newblock {\em Ann. of Math. (2)}, 183(1):95--176, 2016.

\bibitem{LiSun}
C.~Li and S.~Sun.
\newblock Conic {K\"ahler-Einstein} metric revisited.
\newblock {\em Comm. Math. Phys.}, 331(3):927--973, 2014.

\bibitem{LiZh}
J.~Liu and X.~Zhang.
\newblock The conical {K\"ahler-Ricci} flow on {Fano} manifolds.
\newblock {\em Adv. Math.}, 307:1324--1371, \newsort{2014}2017.

\bibitem{LiZh16}
J.~Liu and X.~Zhang.
\newblock The conical {K\"ahler-Ricci} flow with weak initial data on {Fano}
  manifold.
\newblock arXiv:1601.00060, \newsort{2016}2016.

\bibitem{MRS}
R.~Mazzeo, Y.A. Rubinstein, and N.~Sesum.
\newblock Ricci flow on surfaces with conic singularities.
\newblock {\em Anal. PDE}, 8(4):839--882, 2015.

\bibitem{Nom16}
R.~Nomura.
\newblock Blow-up behavior of the scalar curvature along the conical
  {K\"ahler-Ricci} flow with finite time singularities.
\newblock arXiv:1607.03004, 2016.

\bibitem{Per3}
G.~Perelman.
\newblock The entropy formula for {Ricci} flow and its geometric applications.
\newblock arXiv:0211159, 2002.

\bibitem{Per2}
G.~Perelman.
\newblock Ricci flow with surgery on three-manifolds.
\newblock arXiv:0303109, 2003.

\bibitem{PSSWa14}
D.H. Phong, J.~Song, J.~Sturm, and X.W. Wang.
\newblock The {Ricci} flow on the sphere with marked points.
\newblock arXiv:1407.1118, 2014.

\bibitem{PSSWa15}
D.H. Phong, J.~Song, J.~Sturm, and X.W. Wang.
\newblock Convergence of the conical {Ricci} flow on {S2} to a soliton.
\newblock arXiv:1503.04488, 2015.

\bibitem{Ru14}
Y.~Rubinstein.
\newblock Smooth and singular {K\"ahler-Einstein} metrics.
\newblock In {\em Geometric and spectral analysis}, volume 630 of {\em Contemp.
  Math.}, pages 45--138. Amer. Math. Soc., Providence, RI, 2014.

\bibitem{Shen14}
L.~Shen.
\newblock {$C^{2,\alpha}$}-estimate for conical {K\"ahler-Ricci} flow.
\newblock arXiv:1412.2420, 2014.

\bibitem{Shen}
L.~Shen.
\newblock Maximal time existence of unnormalized conical {K\"ahler-Ricci} flow.
\newblock arXiv:1411.7284, 2014.

\bibitem{SoTi09}
J.~Song and G.~Tian.
\newblock The {K\"ahler-Ricci} flow through singularities.
\newblock {\em Invent. Math.}, 207(2):519 -- 595, 2017.

\bibitem{SoWa12}
J.~Song and X.~Wang.
\newblock The greatest {Ricci} lower bound, conical {Einstein} metrics and the
  {Chern} number inequality.
\newblock {\em Geom. Topol.}, 20(1):49--102, 2016.

\bibitem{SoWe11b}
J.~Song and B.~Weinkove.
\newblock \newsort{1}{The} {K\"ahler-Ricci} flow on {Hirzebruch} surfaces.
\newblock {\em J. Reine Angew.}, 659:141--168, 2011.

\bibitem{SoWe11}
J.~Song and B.~Weinkove.
\newblock \newsort{2}{C}ontracting exceptional divisors by the {K\"ahler-Ricci}
  flow.
\newblock {\em Duke Math. J.}, 162(2):367--415, 2011.

\bibitem{SoWe11a}
J.~Song and B.~Weinkove.
\newblock Contracting exceptional divisors by the {K\"ahler-Ricci} flow, {II}.
\newblock {\em Proc. Lond. Math. Soc.}, 108(6):1529--1561, 2014.

\bibitem{Ti96}
G.~Tian.
\newblock {K\"ahler-Einstein} metrics on algebraic manifolds.
\newblock In {\em Transcendental methods in algebraic geometry (Cetraro,
  1994)}, volume 1646 of {\em Lecture Notes in Math.}, pages 143--185.
  Springer, Berlin, 1996.

\bibitem{Ti}
G.~Tian.
\newblock K-stability and {K\"ahler-Einstein} metrics.
\newblock {\em Comm. Pure Appl. Math.}, 68(7):1085--1156, 2015.

\bibitem{TiZh}
G.~Tian and Z.~Zhang.
\newblock On the {K\"ahler-Ricci} flow on projective manifolds of general type.
\newblock {\em Chi. Ann. of Math.}, 27(2):179--192, 2006.

\bibitem{TiZh16}
G.~Tian and Z.~Zhang.
\newblock Convergence of {K\"ahler-Ricci} flow on lower-dimensional algebraic
  manifolds of general type.
\newblock {\em Int. Math. Res. Not. IMRN}, 2016(21):6493--6511, 2016.

\bibitem{troy}
M.~Troyanov.
\newblock Prescribing curvature on compact surfaces with conic singularities.
\newblock {\em Trans. Amer. Math. Soc.}, 324:793--821, 1991.

\bibitem{Ts}
H.~Tsuji.
\newblock Existence and degeneration of {K\"ahler-Einstein} metrics on minimal
  algebraic varieties of general type.
\newblock {\em Math. Ann.}, 281:123--133, 1988.

\bibitem{yau}
S.T. Yau.
\newblock On the {Ricci} curvature of a compact {K\"ahler} manifold and the
  complex {Monge-Amp\`ere} equation. {I}.
\newblock {\em Comm. Pure Appl. Math}, 31(3):339--411, 1978.

\bibitem{Yin10}
H.~Yin.
\newblock Ricci flow on surfaces with conical singularities.
\newblock {\em J. Geom. Anal.}, 20:970--995, 2010.

\bibitem{Yin13}
H.~Yin.
\newblock Ricci flow on surfaces with conical singularities, {II}.
\newblock arXiv:1305.4355, 2013.

\bibitem{YZh16}
Y.~Zhang.
\newblock A note on conical {K\"ahler-Ricci} flow on minimal elliptic
  {K\"ahler} surfaces.
\newblock arXiv:1610.09880, 2016.

\end{thebibliography}

\end{document}